\documentclass[12pt]{amsart}

\usepackage{amsmath,amssymb,enumitem,verbatim,stmaryrd,xcolor,microtype,graphicx,mathscinet,mathtools,dirtytalk}
\usepackage[T1]{fontenc}
\usepackage[utf8]{inputenc}
\usepackage[english]{babel}
\usepackage[top=3.0cm,bottom=3.0cm,left=3.2cm,right=3.2cm]{geometry}
\usepackage[bookmarksdepth=2,linktoc=page,colorlinks,linkcolor={red!80!black},citecolor={red!80!black},urlcolor={blue!80!black}]{hyperref}
\usepackage{tikz}\usetikzlibrary{matrix,arrows,decorations.markings}
\usepackage{tikz-cd}
\usepackage[all]{xy}%\CompileMatrices\CompilePrefix{xymatrix/diagram}
\usepackage[mathcal]{euscript}                               % use the euscript caligraphic font with \mathscr{...}
\usepackage{mathptmx}                                        % times font
\usepackage[backgroundcolor=yellow,linecolor=yellow,textsize=footnotesize]{todonotes} \makeatletter \providecommand \@dotsep{5} \def\listtodoname{List of Todos} \def\listoftodos{\@starttoc{tdo}\listtodoname} 
 \makeatother %\todo{} for margin notes, supress in pdf with option [disable]
\usepackage{etoolbox}
\makeatletter
\patchcmd{\@startsection}{\@afterindenttrue}{\@afterindentfalse}{}{}             %omit indentation of the first paragraph of a section
\patchcmd{\part}{\bfseries}{\bfseries\LARGE}{}{}
\patchcmd{\section}{\scshape}{\bfseries}{}{}\renewcommand{\@secnumfont}{\bfseries} %boldface no smallcaps section and subsection titles with numbers
\patchcmd{\@settitle}{\uppercasenonmath\@title}{\large}{}{}
\patchcmd{\@setauthors}{\MakeUppercase}{}{}{}
\addto{\captionsenglish}{} %boldface no smallcaps Abstract
\addto{\captionsenglish}{} %bolface no smallcaps Figure
\addto{\captionsenglish}{} %bolface no smallcaps Table
\makeatother

\usepackage{fancyhdr}

\theoremstyle{plain}
\newtheorem{thm}{Theorem}[section]

\newtheorem{lemma}[thm]{Lemma}

\newtheorem{thmA}{Theorem}  %alphabetic theorem counter: Theorem A, Theorem B, ...
\newtheorem*{thm*}{Theorem}

\theoremstyle{definition}

\newtheorem{rem}[thm]{Remark}
\newtheorem*{rem*}{Remark}
\newtheorem{ex}[thm]{Example}

\theoremstyle{remark}

\DeclareRobustCommand{\gobblefour}[5]{}    % Command \SkipTocEntry for surpressing a section title in TOC

\DeclareFontFamily{OT1}{pzc}{}                                % Script font for small caligraphic letter, like in \cMat
\DeclareFontShape{OT1}{pzc}{m}{it}{<-> s * [1.10] pzcmi7t}{}
\DeclareMathAlphabet{\mathpzc}{OT1}{pzc}{m}{it}
\DeclareSymbolFont{sfoperators}{OT1}{bch}{m}{n} \DeclareSymbolFontAlphabet{\mathsf}{sfoperators} \makeatletter\def\operator@font{\mathgroup\symsfoperators}\makeatother % different font for math operators
\DeclareSymbolFont{cmletters}{OML}{cmm}{m}{it}              
\DeclareSymbolFont{cmsymbols}{OMS}{cmsy}{m}{n}
\DeclareSymbolFont{cmlargesymbols}{OMX}{cmex}{m}{n}
\DeclareMathSymbol{\myjmath}{\mathord}{cmletters}{"7C}     \let\jmath\myjmath %Defining the missing commands: \jmath, \amalg and \coprod
\DeclareMathSymbol{\myamalg}{\mathbin}{cmsymbols}{"71}     
\DeclareMathSymbol{\mycoprod}{\mathop}{cmlargesymbols}{"60}
\DeclareMathSymbol{\myalpha}{\mathord}{cmletters}{"0B}     \let\alpha\myalpha %Greek letters from Computer Modern since the Greek letters from mathptmx are too large
\DeclareMathSymbol{\mybeta}{\mathord}{cmletters}{"0C}      \let\beta\mybeta
\DeclareMathSymbol{\mygamma}{\mathord}{cmletters}{"0D}     \let\gamma\mygamma
\DeclareMathSymbol{\mydelta}{\mathord}{cmletters}{"0E}     \let\delta\mydelta
\DeclareMathSymbol{\myepsilon}{\mathord}{cmletters}{"0F}   \let\epsilon\myepsilon
\DeclareMathSymbol{\myzeta}{\mathord}{cmletters}{"10}      \let\zeta\myzeta
\DeclareMathSymbol{\myeta}{\mathord}{cmletters}{"11}       \let\eta\myeta
\DeclareMathSymbol{\mytheta}{\mathord}{cmletters}{"12}     \let\theta\mytheta
\DeclareMathSymbol{\myiota}{\mathord}{cmletters}{"13}      \let\iota\myiota
\DeclareMathSymbol{\mykappa}{\mathord}{cmletters}{"14}     \let\kappa\mykappa
\DeclareMathSymbol{\mylambda}{\mathord}{cmletters}{"15}    \let\lambda\mylambda
\DeclareMathSymbol{\mymu}{\mathord}{cmletters}{"16}        \let\mu\mymu
\DeclareMathSymbol{\mynu}{\mathord}{cmletters}{"17}        \let\nu\mynu
\DeclareMathSymbol{\myxi}{\mathord}{cmletters}{"18}        \let\xi\myxi
\DeclareMathSymbol{\mypi}{\mathord}{cmletters}{"19}        \let\pi\mypi
\DeclareMathSymbol{\myrho}{\mathord}{cmletters}{"1A}       \let\rho\myrho
\DeclareMathSymbol{\mysigma}{\mathord}{cmletters}{"1B}     \let\sigma\mysigma
\DeclareMathSymbol{\mytau}{\mathord}{cmletters}{"1C}       \let\tau\mytau
\DeclareMathSymbol{\myupsilon}{\mathord}{cmletters}{"1D}   \let\upsilon\myupsilon
\DeclareMathSymbol{\myphi}{\mathord}{cmletters}{"1E}       \let\phi\myphi
\DeclareMathSymbol{\mychi}{\mathord}{cmletters}{"1F}       \let\chi\mychi
\DeclareMathSymbol{\mypsi}{\mathord}{cmletters}{"20}       \let\psi\mypsi
\DeclareMathSymbol{\myomega}{\mathord}{cmletters}{"21}     \let\omega\myomega
\DeclareMathSymbol{\myvarepsilon}{\mathord}{cmletters}{"22}\let\varepsilon\myvarepsilon
\DeclareMathSymbol{\myvartheta}{\mathord}{cmletters}{"23}  \let\vartheta\myvartheta
\DeclareMathSymbol{\myvarpi}{\mathord}{cmletters}{"24}     \let\varpi\myvarpi
\DeclareMathSymbol{\myvarrho}{\mathord}{cmletters}{"25}    \let\varrho\myvarrho
\DeclareMathSymbol{\myvarsigma}{\mathord}{cmletters}{"26}  \let\varsigma\myvarsigma
\DeclareMathSymbol{\myvarphi}{\mathord}{cmletters}{"27}    \let\varphi\myvarphi

\DeclareMathOperator{\Poly}{Poly}

\DeclareMathOperator{\hypersum}{\,\raisebox{-2.2pt}{\larger[2]{$\boxplus$}}\,}
\DeclareMathOperator{\hyperprod}{\raisebox{-1.4pt}{\scalebox{1.5}{$\boxdot$}}\,}

\newcommand\C{{\mathbb C}}

\newcommand\N{{\mathbb N}}

\newcommand\R{{\mathbb R}}
\renewcommand\S{{\mathbb S}}
\newcommand\T{{\mathbb T}}

\newcommand\cP{{\mathcal P}}

\newcommand\sign{\textup{sign}}

\renewcommand\geq{\geqslant}
\renewcommand\leq{\leqslant}

\newcommand{\hyperplus}{\mathrel{\,\raisebox{-1.1pt}{\larger[-0]{$\boxplus$}}\,}}
\newcommand{\hyperdot}{\,\raisebox{-0.3pt}{\larger[+1]{$\boxdot$}}\,}
\newdir{ >}{{}*!/-5pt/@^{(}}\newcommand{\arincl}[1]{\ar@{ >->}@<-0,0ex>#1} %inclusion arrow for xy-matrix with better spacing

\title{Factorizations of tropical and sign polynomials}
 
\author{Alexander Agudelo}

\address{\rm Alexander Agudelo, Instituto Nacional de Matem\'atica Pura e Aplicada, Rio de Janeiro, Brazil}
\email{alagudel@gmail.com}

\author{Oliver Lorscheid}
\address{\rm Oliver Lorscheid, Instituto Nacional de Matem\'atica Pura e Aplicada, Rio de Janeiro, Brazil}
\email{oliver@impa.br}
\thanks{The authors thank the Max Planck Institute for Mathematics in Bonn that has hosted them during the time of writing this manuscript.}

\begin{document}

\begin{abstract} 
 In this text, we study factorizations of polynomials over the tropical hyperfield and the sign hyperfield, which we call \emph{tropical polynomials} and \emph{sign polynomials}, respectively. We classify all irreducible polynomials in either case. We show that tropical polynomials factor uniquely into irreducible factors, but that unique factorization fails for sign polyomials. We describe division algorithms for tropical and sign polynomials by linear terms that correspond to roots of the polynomials.
\end{abstract}

\maketitle
%{\ \vspace{-235pt}\\ \flushright\tiny\bf\ \today\\ First complete draft\\}\vspace{200pt}

%%%%%%%%%%%%%%%%%%%%%%%%%%%%%%%%%%%%%%%%%%%%%%%%%%%%%%%%%%%%%%%%%%%%%%%%%%%%%%%%%%%%%%%%%%%%%%%%%%%%%%%%%%%%%%%%%%%%%%%%%%%%%%%%%%%%%%%%%%%%%%%%%%%%%%%%%%%%%%%%%%%%%
%%%%%%%%%%%%%%%%%%%%%%%%%%%%%%%%%%%%%%%%%%%%%%%%%%%%%%%%%%%%%%%%%%%%%%%%%%%%%%%%%%%%%%%%%%%%%%%%%%%%%%%%%%%%%%%%%%%%%%%%%%%%%%%%%%%%%%%%%%%%%%%%%%%%%%%%%%%%%%%%%%%%%

\section*{Introduction}
\label{introduction}

Hyperfields were introduced by Marc Krasner in 1956 in his paper \cite{Krasner56} as a generalization of fields by allowing the addition to be multi-valued. Since then a considerable amount of literature on hyperfields has built up, but, still, the notion of a hyperfield stayed largely in the shadow of mainstream mathematics until around a decade ago when the works \cite{Viro11} of Viro and \cite{Connes-Consani11} of Connes and Consani showed the potential of hyperfields for tropical and arithmetic geometry. More recently, Baker and Bowler have demonstrated in \cite{Baker-Bowler17} the relevance of hyperfields for matroid theory. 

The joint paper \cite{Baker-Lorscheid18b} of Baker and the second author of this text provides additional evidence for the usefulness of this concept: a study of roots of polynomials, and their multiplicities, over hyperfields leads to a simultaneous proof of Newton's polygon rule and Descartes' rule of signs. The hyperfields that underlie these rules are the tropical hyperfield $\T$ and the sign hyperfield $\S$, respectively. 

In this paper, we complement the theory from \cite{Baker-Lorscheid18b} by some results on the factorization of polynomials over $\T$ and $\S$, which we call \emph{tropical polynomials} and \emph{sign polynomials}, respectively. Before we turn to a description of our findings, we introduce the two main actors of our text.

\subsection*{The tropical and the sign hyperfield}

The tropical hyperfield $\T$ is the set $\R_{\geq0}$ of non-negative real numbers together with the usual multiplication and with the hyperaddition given by
\[
 a_1 \hyperplus \dotsb \hyperplus a_n \ = \ \begin{cases} \big\{\max\{a_i\}\big\}  &\text{if $a_j<a_i$ for all $i\neq j$}; \\ [0,a_i] &\text{if $a_k\leq a_i=a_j$ for some $i\neq j$ and all $k$.} \end{cases}
\]
%\[
% a \hyperplus b \ = \ \begin{cases} \big\{\max\{a,b\}\big\}  &\text{if }a\neq b; \\ [0,a] &\text{if }a=b. \end{cases}
%\]
%In other words, we have $c\in a\hyperplus b$ if and only if the maximum among $a$, $b$ and $c$ appears twice. Note that $-a=a$ for every $a\in\T$.

The \emph{sign hyperfield $\S$} is the set $\{0,1,-1\}$ together with the obvious multiplication and with the hyperaddition given by 
\[
 a_1 \hyperplus \dotsb \hyperplus a_n \ = \ \begin{cases} \{a_i\} &\text{if $a_k\in\{0,a_i\}$ for all $k$}; \\ \S &\text{if both $1$ and $-1$ are in $\{a_1,\dotsc,a_n\}$}. \end{cases}
\]

\subsection*{Factorization of polynomials}

Let $F$ be a hyperfield---the reader might want to think of $F$ as one of $\T$ or $\S$. A polynomial (of degree $n$) over $F$ is an expression $p=c_nT^n+\dotsc+c_1T+c_0$ with $c_i\in F$ and $c_n\neq 0$ unless $n=0$. Given two polynomials $p=\sum c_iT^i$ and $q=\sum d_iT^i$ over $F$, we define their hyperproduct as the set
\[\textstyle
 p \hyperdot q \ = \ \big\{ \, \sum e_iT^i \, \big| \, e_i\in\underset{k+l=i}\hypersum c_kd_l \, \big\}
\]
of polynomials over $F$. We define recursively the hyperproduct of $n$ polynomials $q_1,\dotsc,q_n$ over $F$ as
\[
 \underset{i=1}{\stackrel{n}\hyperprod} q_i \ = \ \bigcup_{p\in\hyperdot_{i=1}^{n-1} q_i} p \hyperdot q_n.
\]
A polynomial $p$ is \emph{irreducible} if its degree is positive and if for all polynomials $q_1$ and $q_2$ such that $p\in q_1\hyperdot q_2$ either $q_1$ or $q_2$ is of degree $0$. A \emph{quotient of $p$ by $q$} is a polynomial $q'$ such that $p\in q\hyperdot q'$. 

A polynomial $p$ has a \emph{unique factorization into irreducibles} if there irreducible polynomials $q_1,\dotsc,q_n$ that are unique up to a permutation and up to multiplication by constant polynomial such that $p\in c\hyperdot \hyperprod_{i=1}^n q_i$ for $c\in F$.

\subsection*{Unique factorization for tropical polynomials and its failure for sign polynomials} 

Our first result is a classification of all irreducible tropical polynomials and the unique factorization over $\T$. The following is Theorem \ref{thm: unique factorization of tropical polynomials}.

\begin{thmA}\label{thmA}
 The irreducible tropical polynomials are precisely the linear tropical polynomials, and every tropical polynomial has a unique factorization into irreducibles.
\end{thmA}

The list of irreducible sign polynomials is as follows, which is Theorem \ref{thm: classification of irreducible sign polynomials}.

\begin{thmA}\label{thmB}
  Up to multiplication by $-1$, the irreducible sign polynomials are $T$, $T-1$, $T+1$ and $T^2+1$.
\end{thmA}

In contrast to tropical polynomials, sign polynomials fail to have unique factorizations in general. For example the sign polynomial $T^3+T^2+T+1$ is contained in both products
\[
 (T+1)\hyperdot(T+1)\hyperdot(T+1) \qquad \text{and} \qquad (T+1)\hyperdot(T^2+1);
\]
cf.\ section \ref{subsection: the failure of unique factorization} for more details. 

\subsection*{Division algorithms} 

A fundamental fact that enters the definition of the multiplicity of a root is that if a polynomial $f$ over a hyperfield $F$ has a root $a\in F$, then $f$ is divisible by $T-a$ (cf.\ \cite[Lemma A]{Baker-Lorscheid18b}). In the case of usual fields, this follows directly from the division algorithm for polynomials. In the case of hyperfields, this algorithm does not work \textit{a priori} due to the ambiguity of the multi-valued addition of the hyperfield. In particular, it happens that there are several quotients of a polynomial by a linear term.

In this text, we describe algorithms for the division tropical and sign polynomials by linear polynomials. These algorithms might be useful for explicit calculations of multiplicities of tropical and sign polynomials, which is of interest for their link to Newton polygons and Descartes' rule of signs. 

The division algorithms for tropical polynomials is follows. Let $p=\sum c_iT^i$ be a tropical polynomial of degree $n$. By Theorem \ref{thmA}, $p$ factors into a unique product $c_n\prod (T+a_i)$ of linear polynomials $T+a_i$ where we assume that $a_1\leq \dotsb\leq a_n$. It follows from the fundamental theorem for the tropical hyperfield (cf.\ \cite[Theorem 4.1]{Baker-Lorscheid18b}) that $\{a_1,\dotsc,a_n\}$ are the roots of $p$, and that the multiplicity $m$ of a root $a$ of $p$ coincides with the number of $a_i$'s that are equal to $a$, i.e.\ 
\[
 a \ = \ a_k \ = \ \dotsc \ = \ a_{k+m-1}
\]
for some $k\in\{1\dotsc,n-m+1\}$ and $a_{k-1}<a_k$ if $k-1\geq1$ as well as $a_{k+m-1}<a_{k+m}$ if $k+m\leq n$. Since the case $a=0$ is trivial, let us assume that $a$ is not zero. The following is Theorem \ref{thm: division algorithm for tropical polynomials}.

\begin{thmA}\label{thmC}
 Define the tropical numbers $d_0,\dotsc,d_{n-1}$ by the following algorithm.
 \begin{enumerate}
  \item[\textup{(1)}] If $k\leq n-m$, then let $d_{n-1}=c_n$. For $i=n-2,\dotsc,k+m-1$, we define (in decreasing order)
        \[
         d_i \ = \ \max\{c_{i+1},\, ad_{i+1}\}.
        \]
  \item[\textup{(2)}] If $k\geq2$, then let $d_0=a^{-1}c_0$. For $i=1,\dotsc,k-2$, we define (in increasing order)
        \[
         d_i \ = \ \max\{a^{-1}c_i,\, a^{-1}d_{i-1}\}.
        \]
  \item[\textup{(3)}] For $i=k-1,\dotsc,k+m-2$, we define
        \[
         d_i \ = \ a_{i+2}\dotsb a_{n}c_n.
        \]
 \end{enumerate}
 Then the polynomial $q=\sum d_iT^i$ is a divisor of $p$ by $T+a$.
\end{thmA}

The division algorithm for sign polynomials can be described more compactly as follows, which is Theorem \ref{thm: division algorithm for sign polynomials}

\begin{thmA}\label{thmD}
 Let $p=\sum c_iT^i$ be a sign polynomial of degree $n$ where $c_0,\dots,c_n\in\S$. Let $a\in\{\pm 1\}$ be a root of $p$. Define 
 \[
  l \ = \ \min\big\{\, i\in\N\,\big|\,c_i\neq 0\,\big\} \qquad \text{and} \qquad k \ = \ \min\big\{\,i\in\N\,\big|\, c_{i+1}= -a^{i+1-l} c_l\, \big\}.
 \]
 Define recursively for $i=n-1,\dotsc,0$ (in decreasing order)
 \begin{align}
  d_i \ &= \ c_{i+1}        && \text{if $c_{i+1}\neq 0$ and $i>k$;}\\
  d_i \ &= \ ad_{i+1}       && \text{if $c_{i+1}=0$ and $i>k$;}\\
  d_i \ &= \ -a^{i+l-1} c_l && \text{if $l\leq i\leq k$;}\\
  d_i \ &= \ 0              && \text{if $0\leq i<l$.}
 \end{align}
  Then $q=\sum d_iT^i$ is a quotient of $p$ by $T-a$.
\end{thmA}

\subsection*{Acknowledgements}
The authors thank Matt Baker for helpful discussions that led to the results of this paper. We thank Gunn Trevor and Liu Ziqi for correcting some wrong statements of an early draft of this text.

%%%%%%%%%%%%%%%%%%%%%%%%%%%%%%%%%%%%%%%%%%%%%%%%%%%%%%%%%%%%%%%%%%%%%%%%%%%%%%%%%%%%%%%%%%%%%%%%%%%%%%%%%%%%%%%%%%%%%%%%%%%%%%%%%%%%%%%%%%%%%%%%%%%%%%%%%%%%%%%%%%%%%%
%%%%%%%%%%%%%%%%%%%%%%%%%%%%%%%%%%%%%%%%%%%%%%%%%%%%%%%%%%%%%%%%%%%%%%%%%%%%%%%%%%%%%%%%%%%%%%%%%%%%%%%%%%%%%%%%%%%%%%%%%%%%%%%%%%%%%%%%%%%%%%%%%%%%%%%%%%%%%%%%%%%%%%

%\section{Polynomials over hyperfields}
%\label{section: polynomials over hyperfields}

%%%%%%%%%%%%%%%%%%%%%%%%%%%%%%%%%%%%%%%%%%%%%%%%%%%%%%%%%%%%%%%%%%%%%%%%%%%%%%%%%%%%%%%%%%%%%%%%%%%%%%%%%%%%%%%%%%%%%%%%%%%%%%%%%%%%%%%%%%%%%%%%%%%%%%%%%%%%%%%%%%%%%%

\section{Hyperfields}
\label{section: hyperfields}

%Hyperfields are a generalization of fields, which is obtained by allowing subsets instead of elements as values of the addition. In the following we repeat in brevity the definitions of hyperfields and morphisms, and we describe the relevant examples for this text. The reader can find more details and examples in \cite{Viro11}. \todo{Add more references!}

%\subsubsection{Definition}
A \emph{hyperfield} is a set $F$ together with a multiplication, i.e.\ a map $\cdot:F\times F\to F$, and with a \emph{hyperaddition}, which is a map $\hyperplus:F\times F\to \cP(F)$ where $\cP(F)$ is the power set of $F$, that satisfies the following axioms:
\begin{enumerate}[label={(HF\arabic*)}]
 \item\label{HF1} There are unique elements $0$ and $1$ of $F$ such that $(F,\cdot,1)$ is a commutative monoid and such that $F^\times=F-\{0\}$ is a group with respect to $\cdot$.
 \item\label{HF2} For all $a,b,c\in F$, we have $a\cdot b\hyperplus a\cdot c=\{a\cdot d|d\in b\hyperplus c\}$.                                      \hfill\emph{(distributive)}
 \item\label{HF3} $(F\hyperplus,0)$ is a commutative hypergroup, i.e.\ we have for all $a,b,c\in F$ that
                  \begin{enumerate}[label={(HG\arabic*)}]
                   \item\label{HG1} $a\hyperplus b$ is not empty;                                                                     \hfill\emph{(non-empty sums)}
                   \item\label{HG2} $a\hyperplus b=b\hyperplus a$;                                                                    \hfill\emph{(commutative)}
                   \item\label{HG3} $a\hyperplus 0=\{a\}$;                                                                            \hfill\emph{(neutral element)}
                   \item\label{HG4} there is a unique $d\in F$ such that $0\in a\hyperplus d$;                                        \hfill\emph{(additive inverse)}
                   \item\label{HG5} $\bigcup \{a\hyperplus d | d\in b\hyperplus c\} = \bigcup\{d\hyperplus c | d\in a\hyperplus b\}$. \hfill\emph{(associative)}
                  \end{enumerate}
\end{enumerate}
In the following, we write $ab$ for $a\cdot b$ and $-a$ for the additive inverse of $a$, i.e.\ $0\in a\hyperplus(-a)$. For $n\geq3$ and $a_1,\dotsc,a_n\in F$, we define recursively the subset
\[
 \underset{i=1}{\stackrel{n}\hypersum} a_i \ = \ \bigcup_{b\in\hyperplus_{i=1}^{n-1} a_i} b \hyperplus a_n,
\]
of $F$, which does not depend on the order of the $a_i$ thanks to associativity and commutativity.

The axioms of a hyperfield imply that $0\cdot a=0$ for all $a\in F$ and that 
\begin{enumerate}[label={(HG\arabic*)}]\addtocounter{enumi}{5}
 \item\label{HG6} $a\in b\hypersum c$ if and only if $-b\in (-a)\hypersum c$                                                                     \hfill\emph{(reversibility)}
\end{enumerate}
for all $a,b,c\in F$.

\subsection{Examples}
A primary example of hyperfields are fields. Namely, given a field $K$, we can define a hyperaddition $\hyperplus$ on $K$ by the rule $a\hyperplus b=\{a+b\}$, which turns $K$ into a hyperfield.

To give some examples of hyperfields that do not come from fields, let us introduce the two main characters of our story: the tropical hyperfield $\T$ and the sign hyperfield $\S$.

The \emph{tropical hyperfield $\T$} is the set $\R_{\geq0}$ of nonnegative real numbers together with their usual multiplication and the hyperaddition defined by the rule
\[
 a \hyperplus b \ = \ \begin{cases} \big\{\max\{a,b\}\big\}  &\text{if }a\neq b; \\ [0,a] &\text{if }a=b. \end{cases}
\]
In other words, we have $c\in a\hyperplus b$ if and only if the maximum among $a$, $b$ and $c$ appears twice. Note that $-a=a$ for every $a\in\T$.

The \emph{sign hyperfield $\S$} is the set $\{0,1,-1\}$ together with the obvious multiplication and with the hyperaddition given by the table
\[
\begin{tabular}{|c||c|c|c|}
 \hline
 \small $\hyperplus$ & \small $0$ & \small $1$ & \small $-1$ \\
 \hline\hline 
 \small $0$ & \small $\{0\}$ & \small $\{1\}$ & \small $\{-1\}$ \\
 \hline 
 \small $1$ & \small $\{1\}$ & \small $\{1\}$ & \small $\{0,1,-1\}$ \\
 \hline 
 \small $-1$ & \small $\{-1\}$ & \small $\{0,1,-1\}$ & \small $\{-1\}$ \\
 \hline 
\end{tabular}
\]

\subsection{Morphisms of hyperfields}
\label{subsubsection: morphisms of hyperfields}

A \emph{morphism between hyperfields $F_1$ and $F_2$} is a map $f:F_1\to F_2$ such that $f(0)=0$, $f(1)=1$, $f(ab)=f(a)f(b)$ and $f(a\hyperplus b)\subset f(a)\hyperplus f(b)$ for all $a,b\in F_1$. Note that the latter property is equivalent with requiring that whenever $b\in \hypersum a_i$ in $F_1$, then $f(b)\in\hypersum f(a_i)$ in $F_2$.

Let us describe the two examples of morphisms of hyperfields that are of interest for our purpose. The first example is that of the sign map $\sign:\R\to \S$ that associates with a nonzero real number $a\in\R$ its sign $\sign(a)=a/|a|$ and that maps $0$ to $0$.

The second example is based on a general fact observed by Viro in \cite{Viro11}. Namely, by identifying a field $K$ with its associated hyperfield and the nonnegative real numbers $\R_{\geq0}$ with the tropical hyperfield $\T$ as sets, a nonarchimedean absolute value $v:K\to\R_{\geq0}$ is the same as a morphism of hyperfields $v:K\to\T$.

%%%%%%%%%%%%%%%%%%%%%%%%%%%%%%%%%%%%%%%%%%%%%%%%%%%%%%%%%%%%%%%%%%%%%%%%%%%%%%%%%%%%%%%%%%%%%%%%%%%%%%%%%%%%%%%%%%%%%%%%%%%%%%%%%%%%%%%%%%%%%%%%%%%%%%%%%%%%%%%%%%%%%%

\section{Polynomials over hyperfields}
\label{section: polynomials}

A \emph{polynomial over a hyperfield $F$} is an expression of the form $p=c_nT^n+\dotsb+c_1T+c_0$ with $c_0,\dotsc,c_{n}\in F$, or, more formally, a sequence $(c_i)_{i\in\N}$ of elements $c_i\in F$ for which $\{i\in\N|c_i\neq 0\}$ is finite. We denote the set of all polynomials over $F$ by $\Poly(F)$. 

Note that for a field $K$, $\Poly(K)$ is equal to the usual polynomial algebra $K[T]$. For reasons explained in \cite[Appendix A]{Baker-Lorscheid18b}, we refrain from the notation $F[T]$ for the set $\Poly(F)$ of polynomials over a hyperfield $F$.

%a polynomial over $K$ considered as a hyperfield is the same as a polynomial over $K$ in the usual sense.

We will identify elements $a$ of $F$ with the constant polynomial $p=a$ over $F$. In particular, we write $0$ for the zero polynomial $p=0$ and $1$ for the constant polynomial $p=1$.

%%%%%%%%%%%%%%%%%%%%%%%%%%%%%%%%%%%%%%%%%%%%%%%%%%%%%%%%%%%%%%%%%%%%%%%%%%%%%%%%%%%%%%%%%%%%%%%%%%%%%%%%%%%%%%%%%%%%%%%%%%%%%%%%%%%%%%%%%%%%%%%%%%%%%%%%%%%%%%%%%%%%%%

\subsection{Hyperproducts}
\label{subsection: hyperproducts}

The multiplication and hyperaddition of a hyperfield $F$ endows the set $\Poly(F)$ of polynomials over $F$ with an additive and a multiplicative structure. The additive structure is the hyperaddition on $\Poly(F)$ that results from the hyperaddition of coefficients, which might not come as a surprise. Since the hyperaddition of polynomials is not of interest for our present purpose, we omit a discussion, but refer the reader to \cite[Appendix A]{Baker-Lorscheid18b} for details.

The multiplicative structure of $\Poly(F)$ is the \emph{hypermultiplication} 
\[
 \hyperdot: \ \Poly(F)\times\Poly(F) \quad \longrightarrow \quad \cP\big(\Poly(F)\big)
\]
that maps a pair of polynomials $p=\sum c_iT^i$ and $q=\sum d_iT^i$ to the subset
\[\textstyle
 p \hyperdot q \ = \ \big\{ \, \sum e_iT^i \, \big| \, e_i\in\underset{k+l=i}\hypersum c_kd_l \, \big\}
\]
of $\Poly(F)$. Note that in the case of a hyperfield coming from a field $F$, $p\hyperdot q=\{pq\}$ is the singleton containing the usual product of $p$ and $q$.

It is easily verified that this hypermultiplication satisfies the following properties in analogy to that of a hyperaddition (cf.\ \ref{HG1}--\ref{HG3} in section \ref{section: hyperfields}): for all $p,q\in\Poly(F)$, we have
\begin{enumerate}[label={(HM\arabic*)}]
 \item\label{HM1} $p\hyperdot q$ is not empty;                                                                      \hfill\emph{(non-empty sums)}
 \item\label{HM2} $p\hyperdot q=q\hyperdot p$;                                                                      \hfill\emph{(commutative)}
 \item\label{HM3} $p\hyperdot 1=\{p\}$.                                                                             \hfill\emph{(neutral element)}
% \item\label{HM4} $a\hyperdot 0=\{0\}$.                                                                             \hfill\emph{(product by zero)}
\end{enumerate}
Similar as for the hyperaddition of a hyperfield, we extend $\hyperdot$ recursively to $n$-fold products by the rule
\[
 \underset{i=1}{\stackrel{n}\hyperprod} q_i \ = \ \bigcup_{p\in\hyperdot_{i=1}^{n-1} q_i} p \hyperdot q_n.
\]

Note that the definition of the $n$-fold product depends on the order of the factors in general since, in contrast to the situation over a field, $\hyperdot$ fails to be associative for some hyperfields. This is, in particular, the case for $\Poly(\T)$ and $\Poly(\S)$, as shown in \cite{Liu19}.

\subsection{The degree}
\label{subsection: degree}

The \emph{degree} of a polynomial $p=c_nT^n+...+c_0$ over a hyperfield $F$ is the largest $k\in\N$ such that $c_k\neq 0$, which we denote by $\deg p$. %The \emph{degree of $p$} is the largest index $\deg p=i$ of a nonzero coefficient $c_i$, and $p$ is \emph{monic} if $c_{\deg p}=1$. 

For $n$ polynomials $q_1,\dotsc, q_n$, we have $\deg p=\sum\deg q_i$ for every $p\in\hyperprod q_i$. In so far, $1\in p\hyperprod q$ implies that $\deg p=\deg q=0$ and thus $p=a$ and $q=a^{-1}$ for some $a\in F^\times$. For every $p=\sum c_nT^i$, we have $p\hyperdot aT^k=\big\{\sum ac_iT^{i+k}\big\}$ and, in particular, $p\hyperdot 0=\{0\}$. We write $-p$ for the unique element $\sum(-c_n)T^n$ in $(-1)\hyperdot p$.

%%%%%%%%%%%%%%%%%%%%%%%%%%%%%%%%%%%%%%%%%%%%%%%%%%%%%%%%%%%%%%%%%%%%%%%%%%%%%%%%%%%%%%%%%%%%%%%%%%%%%%%%%%%%%%%%%%%%%%%%%%%%%%%%%%%%%%%%%%%%%%%%%%%%%%%%%%%%%%%%%%%%%%

\subsection{Factorizations}
\label{subsection: Factorization}

We say that $p$ and $q$ are \emph{associated}, and write $p\sim q$, if $p\in a\hyperdot q$ for some $a\in F^\times$. Note that $\sim$ is an equivalence relation. A polynomial $p$ is \emph{monic} if $c_{\deg p}=1$. We conclude that for every nonzero polynomial $p$ there is a unique monic polynomial $q$ with $p\sim q$.

%For $n\geq3$ and $q_1,\dotsc,q_n\in \Poly(F)$, we define recursively the subset
%\[
% \underset{i=1}{\stackrel{n}\hyperprod} q_i \ = \ \bigcup_{p\in\hyperdot_{i=1}^{n-1} q_i} p \hyperdot q_n,
%\]
%of $\Poly(F)$. If \ref{HM4} holds, then this definition does not depend on the order of the $q_i$.

Let $p,q_1,\dotsc,q_n \in\Poly(F)$ be polynomials over $F$. We say that \emph{$p$ factors into the product of $q_1,\dotsc,q_n$} if $p\in\hyperprod q_i$.

We can extend a morphism $f:F_1\to F_2$ to a map $f:\Poly(F_1)\to\Poly(F_2)$ between polynomials: given a polynomial $p=\sum c_iT^i$ over $F_1$, we define $f(p)=\sum f(c_i)T^i$.

\begin{lemma}\label{lemma: factorizations are preserved under morphisms}
 Let $f:F_1\to F_2$ be a morphism of hyperfields, $n\geq2$ and $p,q_1,\dotsc,q_n \in\Poly(F_1)$ such that $p\in\hyperprod q_i$. Then $f(p)\in \hyperprod f(q_i)$.
\end{lemma}

\begin{proof}
 Let $p=\sum c_jT^j$ and $q_i=\sum d_{i,j}T^j$. We prove the claim by induction on $n\geq2$.
 
 If $n=2$, then $p\in q_1\hyperdot q_2$ means that $c_i\in \hypersum_{k+l=i} d_{1,k}d_{2,l}$ in $F_1$. Since $f$ is a morphism of hyperfields, we have $f(c_i)\in \hypersum_{k+l=i} f(d_{1,k})f(d_{2,l})$ in $F_2$, and thus $f(p)\in f(q_1)\hyperdot f(q_2)$ as claimed.
 
 If $n>2$, then $r\in\hyperdot_{i=1}^{n-1} q_i$ and the inductive hypothesis imply that $f(r)\in \hyperdot_{i=1}^{n-1} f(q_i)$. This and the case $n=2$ show that $p\in\hyperprod q_i$ implies that 
 \[
  f(p) \ \in \ f\Big(\hspace{-10pt}\bigcup_{r\in\hyperdot_{i=1}^{n-1} q_i} r \hyperdot q_n \Big) \ = \ \bigcup_{r\in\hyperdot_{i=1}^{n-1} q_i} f\big(r \hyperdot q_n \big)  \ \subset \ \bigcup_{f(r)\in\hyperdot_{i=1}^{n-1} f(q_i)} f(r) \hyperdot f(q_n),
 \]
 which establishes the claim of the lemma.
%  \ \subset \ \bigcup_{f(r)\in\hyperdot_{i=1}^{n-1} f(q_i)} f\big(r \hyperdot q_n \big) 
% where we employ the case $n=2$ for the last inclusion. This concludes the proof of the lemma. 
% By the definition of the hyperproduct of polynomials, we have $p\in \hyperprod q_i$ if and only if $c_j\in\hypersum_{\sum k_i=j} \big(d_{1,k_1}\dotsb d_{n,k_n}\big)$ for all $j$. Since a morphism preserves products and hypersums, we conclude that 
% \[
%  f(c_j)\ \in \ \underset{\sum k_i=j}\hypersum \big(f(d_{1,k_1})\dotsb f(d_{n,k_n})\big)
% \]
% for all $j$, and thus $f(p)\in\hyperprod f(q_i)$.
\end{proof}

%%%%%%%%%%%%%%%%%%%%%%%%%%%%%%%%%%%%%%%%%%%%%%%%%%%%%%%%%%%%%%%%%%%%%%%%%%%%%%%%%%%%%%%%%%%%%%%%%%%%%%%%%%%%%%%%%%%%%%%%%%%%%%%%%%%%%%%%%%%%%%%%%%%%%%%%%%%%%%%%%%%%%%

\subsection{Irreducible polynomials}
\label{subsection: irreducible polynomials}

We say that a polynomial $p=\sum c_iT^i$ over a hyperfield $F$ is \emph{irreducible} if $\deg p\geq 1$ and if for every factorization $p\in q_1\hyperdot q_2$, we have $p\sim q_1$ or $p\sim q_2$. Note that if $p\sim q$, then $p$ is irreducible if and only if $q$ is irreducible. 

Since $p\in q_1\hyperdot q_2$ implies $\deg p=\deg q_1+\deg q_2$, we have $p\sim q_1$ if and only if $\deg p=\deg q_1$, or if, equivalently, $q_2=c_0\in F^\times$ is a constant nonzero polynomial. It follows that every linear polynomial $p=c_1T+c_0$ (with $c_1\neq 0$) is irreducible. In Lemma \ref{lemma: irreducibility of quadratic polynomials}, we give an irreducibility criterion for quadratic and cubic polynomials. 

The following fact was pointed out to us by Trevor Gunn.

\begin{lemma}\label{lemma: irreducibility for factorizations into n factors}
 Let $p\in\Poly(F)$ be irreducible, $n\geq2$ and $p\in \hyperprod_{i=1}^n q_i$ a factorization. Then there is an $i\in\{1,\dotsc,n\}$ such that $p\sim q_i$.
\end{lemma}

\begin{proof}
 We prove the claim by induction on $n$. The case $n=2$ follows by the definition of irreducibility.
 
 Assume that $n>2$. By the definition of $\hyperprod_{i=1}^n q_i$, there is an $r$ in $\hyperprod_{i=1}^{n-1} q_i$ such that $p\in r\hyperdot q_n$. Since $p$ is irreducible, we have $p\sim r$ or $p\sim q_n$. If $p\sim q_n$, then there is nothing to prove. If $p\sim r$, then $r$ is also irreducible. By the inductive hypothesis, applied to $r$, we have $p\sim r\sim q_i$ for some $i\in\{1,\dotsc,n-1\}$, which completes the proof of the lemma.
\end{proof}

We say that $\Poly(F)$ \emph{has the unique factorization property} if for any two factorizations $p\in q_1\hyperdot \dotsb\hyperdot q_n$ and $p\in q'_1\hyperdot \dotsb\hyperdot q'_m$ into irreducible factors $q_1,\dotsc, q_n,q'_1,\dotsc,q'_m$, we have $n=m$ and $q_i\sim q'_{\sigma(i)}$ for some permutation $\sigma$ of $\{1,\dotsc,n\}$.

%Note that if $\deg q_j=0$ for all $q_j$ but one $q_i$, then it follows that $q_j=a_j\in F$ for $j\neq i$ and $p\in a\hyperdot q_i$ for $a=\prod_{j\neq i}a_j$. Since $\deg\big(\hyperprod q_j\big)=\sum\deg q_j$, it follows that every linear polynomial $p=c_1T+c_0$ (with $c_1\neq 0$) is irreducible. In Lemma \ref{lemma: irreducibility of quadratic polynomials}, we give an irreducibility criterion for quadratic and cubic polynomials.

We conclude with the following implication of unique factorization on a weakened form of associativity. By definition, we have $p\hyperdot q\hyperdot r=(p\hyperdot q)\hyperdot r$. In contrast, $p\hyperdot (q\hyperdot r)$ must be read as $\bigcup\{p\hyperdot s | s\in q\hyperdot r\}$.

\begin{lemma}\label{lemma: unique factorization implies associativity for irreducible polynomials}
 If $\Poly(F)$ has the unique factorization property, then $(p\hyperdot q)\hyperdot r=p\hyperdot (q\hyperdot r)$ for all irreducible $p,q,r\in\Poly(F)$.
\end{lemma}

\begin{proof}
 Let $p,q,r\in\Poly(F)$ irreducible polynomials. Since $\Poly(F)$ has the unique factorization property, we have $s\in (p\hyperdot q)\hyperdot r$ if and only if $s\in (q\hyperdot r)\hyperdot p$. Using the commutativity of $\hyperdot$, we find the desired equality $(p\hyperdot q)\hyperdot r=(q\hyperdot r)\hyperdot p=p\hyperdot (q\hyperdot r)$. 
\end{proof}

\begin{comment}
\begin{rem}\label{rem: irreducibility for two term factorizations}
 It is not clear to us if irreducibility can be tested on factorizations in two factors in general. Thus we suggest the following problem: for which hyperfields $F$ is it true that a polynomial $p\in\Poly(F)$ is irreducible as soon as every factorization $p\in q_1\hyperdot q_2$ into two factors $q_1$ and $q_2$ implies that either $p\sim q_1$ or $p\sim q_2$?
 
 Obviously, this is known to hold for fields. This property also follows immediately for the tropical hyperfield $\T$ since irreducible polynomials are linear (Theorem \ref{thm: unique factorization of tropical polynomials}) and since there is a algorithm (Theorem \ref{thm: division algorithm for tropical polynomials}) for dividing out linear terms. With a bit of work, one can deduce this property from our results also for the sign hyperfield $\S$.
\end{rem}
 \end{comment}

%%%%%%%%%%%%%%%%%%%%%%%%%%%%%%%%%%%%%%%%%%%%%%%%%%%%%%%%%%%%%%%%%%%%%%%%%%%%%%%%%%%%%%%%%%%%%%%%%%%%%%%%%%%%%%%%%%%%%%%%%%%%%%%%%%%%%%%%%%%%%%%%%%%%%%%%%%%%%%%%%%%%%%

\subsection{Roots}
\label{subsection: roots}

Let $a\in F$ and $p=\sum c_iT^i$ be a polynomial over a hyperfield $F$. We say that \emph{$a$ is a root of $p$}, and write $0\in p(a)$, if $0\in \hypersum c_ia^i$. Alternatively, we can characterize roots in terms of the following fact, which is Lemma A in \cite{Baker-Lorscheid18b}.

\begin{lemma}\label{lemma: characterization of roots}
 Let $a\in F$ and $p\in\Poly(F)$. Then $0\in p(a)$ if and only if there exists a $q\in\Poly(F)$ such that $p\in (T-a)\hyperdot q$.
\end{lemma}

Note that if $p=\sum c_iT^i$ and $q=\sum d_iT^i$, then the relation $p\in (T-a)\hyperdot q$ is equivalent with $n=\deg p=1+\deg q$ and the relations 
\[
 c_0=-ad_0, \quad c_i \in  (-ad_i) \hyperplus d_{i-1} \quad \text{for } i=1,\dots, n-1, \quad \text{and} \quad c_{n}=d_{n-1}.
\]
\begin{ex}
 In the case of a field $K$, we have $0\in p(a)$ in the hyperfield sense if and only if $0=p(a)$ in the usual sense, and we have $p\in (T-a)\hyperdot q$ in the hyperfield sense if and only if $p=(T-a)q$ in the usual sense. 
\end{ex}

As an immediate consequence of Lemma \ref{lemma: characterization of roots}, we see that an irreducible polynomial of degree at least $2$ cannot have any roots. For quadratic and cubic polynomials, this implication can be reversed.

\begin{lemma}\label{lemma: irreducibility of quadratic polynomials}
 Let $p$ be a polynomial over $F$ of degree $2$ or $3$. Then $p$ is irreducible if and only if $p$ does not have a root in $F$.
\end{lemma}

\begin{proof}
 As noted before, if $p$ is irreducible, it cannot have any roots. If $p$ is not irreducible, then $p\in q_1\hyperdot q_2$ for a linear polynomial $q_1$ and polynomial $q_2$ of degree $1$ or $2$. After multiplying $q_2$ with the leading coefficient $d_1$ of $q_1$, and $q_1$ by its inverse $d_1^{-1}$, we can assume that $q_1$ is of the form $T-a$. Thus we have $p\in (T-a)\hyperdot q_2$, i.e.\ $a$ is a root of $p$ by Lemma \ref{lemma: characterization of roots}.
% As noted before, if $p$ is irreducible, it cannot have any roots. If $p$ is not irreducible, then $p\in q_1\hyperdot q_2$ for two linear polynomials $q_1$ and $q_2$. After multiplying $q_2$ with the leading coefficient $d_1$ of $q_1$, and $q_1$ by its inverse $d_1^{-1}$, we can assume that $q_1$ is of the form $T-a$. Thus we have $p\in (T-a)\hyperdot q_2$, and $a$ is a root of $p$ by Lemma \ref{lemma: characterization of roots}.
\end{proof}

\begin{comment}
\begin{rem}
 It is not clear to us if it is true that a cubic polynomial without roots is necessarily irreducible, which is true for fields. The problem is that there could be a factorization of a cubic polynomial into three linear terms, but none into two polynomials of positive degree. A positive answer to the question about two term factorizations from Remark \ref{rem: irreducibility for two term factorizations} would imply that a cubic polynomial without roots is irreducible.
\end{rem}
\end{comment}

%%%%%%%%%%%%%%%%%%%%%%%%%%%%%%%%%%%%%%%%%%%%%%%%%%%%%%%%%%%%%%%%%%%%%%%%%%%%%%%%%%%%%%%%%%%%%%%%%%%%%%%%%%%%%%%%%%%%%%%%%%%%%%%%%%%%%%%%%%%%%%%%%%%%%%%%%%%%%%%%%%%%%%
%%%%%%%%%%%%%%%%%%%%%%%%%%%%%%%%%%%%%%%%%%%%%%%%%%%%%%%%%%%%%%%%%%%%%%%%%%%%%%%%%%%%%%%%%%%%%%%%%%%%%%%%%%%%%%%%%%%%%%%%%%%%%%%%%%%%%%%%%%%%%%%%%%%%%%%%%%%%%%%%%%%%%%

\section{Factorizations of tropical polynomials}
\label{section: tropical polynomials}

A \emph{tropical polynomial} is a polynomial over the tropical hyperfield $\T$. In this section, we use the fundamental theorem for the tropical hyperfield to establish the unique factorization of tropical polynomials into linear polynomials, and we describe a division algorithm.

%%%%%%%%%%%%%%%%%%%%%%%%%%%%%%%%%%%%%%%%%%%%%%%%%%%%%%%%%%%%%%%%%%%%%%%%%%%%%%%%%%%%%%%%%%%%%%%%%%%%%%%%%%%%%%%%%%%%%%%%%%%%%%%%%%%%%%%%%%%%%%%%%%%%%%%%%%%%%%%%%%%%%%

%\subsection{The fundamental theorem for the tropical hyperfield}
%\label{subsection: fundamental theorem for the tropical hyperfield}

%%%%%%%%%%%%%%%%%%%%%%%%%%%%%%%%%%%%%%%%%%%%%%%%%%%%%%%%%%%%%%%%%%%%%%%%%%%%%%%%%%%%%%%%%%%%%%%%%%%%%%%%%%%%%%%%%%%%%%%%%%%%%%%%%%%%%%%%%%%%%%%%%%%%%%%%%%%%%%%%%%%%%%

\subsection{Unique factorization for tropical polynomials}
\label{subsection: unique factorization for tropical polynomials}

The fact that every polynomial function on the tropical line is piecewise linear can be expressed by saying that every polynomial function over the tropical numbers factors uniquely into linear functions. This is sometimes called the \emph{fundamental theorem of tropical algebra}. 

This result is reflected by the following variant for the tropical hyperfield, which we call the \emph{fundamental theorem for the tropical hyperfield}.

Let $p=\sum c_iT^i$ be a monic polynomial of degree $n$ over $\T$ and let $a_1,\dotsc,a_n\in\T$. Then we have 
\[
 p\in \underset{i=1}{\stackrel{n}\hyperprod}(T+a_i) \quad \text{if and only if} \quad c_{i} \in \underset{e_{i+1}<\dotsb<e_{n}}\hypersum a_{e_{i+1}}\dotsb a_{e_n} \quad \text{for all }i=0,\dotsc,n-1. 
\]
If $a_1\leq \dotsb \leq a_n$, then this is equivalent to the conditions that $c_i\leq a_{i+1}\dotsb a_n$ for all $i=0,\dotsc,n-1$, with equality holding if $a_{i}<a_{i+1}$. The following is Theorem 4.1 in \cite{Baker-Lorscheid18b}.

\begin{thm}[Fundamental theorem for the tropical hyperfield]\label{thm: fundamental theorem for the tropical hyperfield}
 Let $p=\sum_{i=0}^n c_i T^i$ be a monic polynomial of degree $n$ over $\T$. Then there is a unique sequence $a_1,\dotsc,a_n\in\T$ with $a_1\leq\dotsb\leq a_n$ such that $p\in \hyperprod(T+a_i)$, and $a\in\T$ is a root of $p$ if and only if $a\in\{a_1,\dotsc, a_n\}$.
\end{thm}

In addition, \cite[Thm.\ 4.1]{Baker-Lorscheid18b} provides an effective way to compute the tropical numbers $a_1,\dotsc,a_n$: they correspond to the slopes of the linear segments of the Newton polygon of $p$; cf.\  section \ref{subsection: an example} for an example. This allows us to formulate a division algorithm for tropical polynomials in section \ref{subsection: division algorithm for tropical polynomials}.

A direct consequence of Theorem \ref{thm: fundamental theorem for the tropical hyperfield} is the unique factorization of tropical polynomials.

\begin{thm}\label{thm: unique factorization of tropical polynomials}
 The irreducible tropical polynomials are precisely the linear tropical polynomials, and $\Poly(\T)$ has the unique factorization property.
\end{thm}

\begin{rem}
 Using the methods of this text, we find the following short argument to prove the fact that every irreducible tropical polynomial is linear. Namely, consider a surjective morphism $v:K\to \T$ from an algebraically closed field $K$ to $\T$. For example, we could take the (exponential) valuation $v_p:\C\{T\}\to \T$ of the field of Puiseux series $K=\C\{T\}$ over $\C$ with real exponents.
 
 For an irreducible tropical polynomial $p=\sum c_iT^i$, we choose elements $\hat c_i\in K$ with $v(\hat c_i)=c_i$. By Lemma \ref{lemma: factorizations are preserved under morphisms}, the polynomial $\hat p=\sum\hat c_iT^i$ is irreducible over $K$. Since $K$ is algebraically closed, $\hat p$ is linear, and so is $p=v(\hat p)$.
\end{rem}

%%%%%%%%%%%%%%%%%%%%%%%%%%%%%%%%%%%%%%%%%%%%%%%%%%%%%%%%%%%%%%%%%%%%%%%%%%%%%%%%%%%%%%%%%%%%%%%%%%%%%%%%%%%%%%%%%%%%%%%%%%%%%%%%%%%%%%%%%%%%%%%%%%%%%%%%%%%%%%%%%%%%%%

\subsection{A division algorithm for tropical polynomials}
\label{subsection: division algorithm for tropical polynomials}

While it is a direct calculation to verify whether $0\in p(a)$ for an element $a$ of $\T$ and a tropical polynomial $p$, it it is not so clear how to find a $q\in\Poly(F)$ that satisfies $p\in (T-a)\hyperdot q$, which exists by Lemma \ref{lemma: characterization of roots}. In the case of a field $K$, this can be done using the usual division algorithm for polynomials over $K$. For the tropical hyperfield, there is a similar, but slightly more involved, algorithm, which we describe in the following.

%Dividing all coefficients by the leading coefficient $c_n$ does not change the roots of $p$, so we can assume without loss of generality that $p$ is monic, i.e.\ $c_n=1$. 

Let $p=\sum c_iT^i$ be a polynomial of degree $n$ over $\T$. By Theorem \ref{thm: fundamental theorem for the tropical hyperfield}, there is a unique sequence $a_1\leq \dotsb\leq a_n$ of tropical numbers such that $c_n^{-1}p\in\hyperprod(T+a_i)$. Since the roots of $c_n^{-1}p$ are the same as the roots of $p$, we conclude that the roots of $p$ are $a_1,\dotsc,a_n$, counted with multiplicities. Fix a root $a\in\{a_1,\dotsc,a_n\}$ of multiplicity $m$, i.e.\ 
\[
 a \ = \ a_k \ = \ \dotsc \ = \ a_{k+m-1}
\]
for some $k\in\{1\dotsc,n-m+1\}$ and $a_{k-1}<a_k$ if $k\geq2$ as well as $a_{k+m-1}<a_{k+m}$ if $k\leq n-m$. If $a=0$, then $c_0=0$ and $q=\sum_{i=0}^{n-1} c_{i+1} T^i$ is the unique polynomial such that $p\in (T-0)\hyperdot q$.

Thus let us assume from here on that $a$ is not zero. We can determine a polynomial $q=\sum d_iT^i$ of degree $n-1$ with $p\in(T+a)\hyperdot q$ by the following recursive definition.
\begin{enumerate}
 \item\label{algo1} If $k\leq n-m$, then let $d_{n-1}=c_n$. For $i=n-2,\dotsc,k+m-1$, we define (in decreasing order)
       \[
        d_i \ = \ \max\{c_{i+1},\, ad_{i+1}\}.
       \]
 \item\label{algo2} If $k\geq2$, then let $d_0=a^{-1}c_0$. For $i=1,\dotsc,k-2$, we define (in increasing order)
       \[
        d_i \ = \ \max\{a^{-1}c_i,\, a^{-1}d_{i-1}\}.
       \]
 \item\label{algo3} For $i=k-1,\dotsc,k+m-2$, we define
       \[
        d_i \ = \ a_{i+2}\dotsb a_{n}c_n.
       \]
\end{enumerate}

\begin{thm}\label{thm: division algorithm for tropical polynomials}
 If $a\neq 0$ is a root of $p$, then the polynomial $q=\sum d_iT^i$ as defined above satisfies $p\in(T+a)\hyperdot q$, i.e.\
 \[
  c_n \ = \ d_{n-1}, \qquad c_0 \ = \ ad_0 \qquad \text{and} \qquad c_i \in (ad_i)\hyperplus d_{i-1} \qquad \text{for} \qquad i=1,\dotsc,n-1.
 \]
\end{thm}

\begin{rem}\label{rem: division algorithm}
 The recursion in step \eqref{algo1} stays in a direct analogy to the division algorithm for polynomials over a field, which is given by the formulas $d_{n-1}=c_n$ and $d_i=c_{i+1}+ad_{i+1}$ where $i$ decreases from $n-2$ to $0$. In the tropical setting, step \eqref{algo1} of the algorithm fails in general to provide the required result if used to define all coefficients of $q$, cf.\ section \ref{subsection: an example}. To achieve $p\in (T+a)\hyperdot q$, one needs to define the coefficients $d_i$ for smaller $i$ in terms of step \eqref{algo2}.
 
 In contrast, the coefficients $d_i$ occurring in step \eqref{algo3} could also be defined by the recursions \eqref{algo1} or \eqref{algo2}---all three definitions yield the same result in this case. We chose for the definition as it is because it is explicit and therefore useful for calculation, and it is this form that we use in the proof of Theorem \ref{thm: division algorithm for tropical polynomials}.
 
 Another facility in calculating the coefficients of $q$ is that whenever $a_{i}<a_{i+1}$, then $c_{i}=a_{i+1}\dotsb a_nc_n$. In the course of the proof of Theorem \ref{thm: division algorithm for tropical polynomials}, we show that $ad_{i}\leq a_{i+1}\dotsb a_nc_n$ for $k+m-1\leq i \leq n-1$ and $d_i\leq a_{i+2}\dotsb a_nc_n$ for $0\leq i \leq k-2$. Thus we have $d_i=c_{i+1}$ if $k+m-1\leq i \leq n-2$ and $a_{i+1}< a_{i+2}$, and we have $d_i=a^{-1}c_i$ if $1\leq i \leq k-2$ and $a_i<a_{i+1}$. This means that the recursive definitions in steps \eqref{algo1} and \eqref{algo2} are only needed if multiple zeros other than $a$ occur. In particular, $q$ can be defined explicitly and is unique if all zeros of $p$ are simple.
 
 If multiple zeros occur, then $q=\sum d_iT^i$ is in general not the unique divisor of $p$ by $T+a$, but it is maximal among all such divisors in the following sense: if $p\in(T+a)\hyperdot q'$ for another polynomial $q'=\sum d_i'T^i$, then $d_i'\leq d_i$. This additional statement is easily derived from the proof of Theorem \ref{thm: division algorithm for tropical polynomials}, but we will forgo to spell out the details.
\end{rem}

%%%%%%%%%%%%%%%%%%%%%%%%%%%%%%%%%%%%%%%%%%%%%%%%%%%%%%%%%%%%%%%%%%%%%%%%%%%%%%%%%%%%%%%%%%%%%%%%%%%%%%%%%%%%%%%%%%%%%%%%%%%%%%%%%%%%%%%%%%%%%%%%%%%%%%%%%%%%%%%%%%%%%%

\subsection{An example}
\label{subsection: an example}

%\begin{ex} \label{ex: division algorithm}

 As an illustration of the division algorithm and of some observations from Remark \ref{rem: division algorithm}, we consider the tropical polynomial $p=T^3+rT^2+T+r$ and $a=r$ where $r>1$. The roots of $p$ can be determined from the Newton polygon of $p$, which is the maximal convex function $\rho:[0,3]\to\R$ with $\rho(i)\leq -\log c_i$ for $i=0,\dotsc,3$. Its graph looks as follows:
 
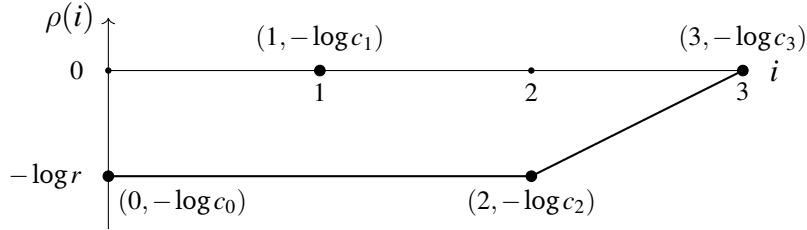
\[
  \begin{tikzpicture}[inner sep=0,x=80pt,y=40pt,font=\footnotesize]
   \filldraw (0,0) circle (1pt);
   \filldraw (2,0) circle (1pt);
   \filldraw (0,-1) circle (2pt);
   \filldraw (1,0) circle (2pt);
   \filldraw (2,-1) circle (2pt);
   \filldraw (3,0) circle (2pt);
   \node at (-0.15,0) {$0$};
   \node at (-0.3,-1) {$-\log r$};
%   \node at (0.2,-0.2) {$0$};
   \node at (1.0,-0.2) {$1$};
   \node at (2.0,-0.2) {$2$};
   \node at (3.0,-0.2) {$3$};
   \node at (0.35,-1.25) {$(0,-\log c_0)$};
   \node at (1,0.3) {$(1,-\log c_1)$};
   \node at (2.0,-1.25) {$(2,-\log c_2)$};
   \node at (3,0.3) {$(3,-\log c_3)$};
   \draw[-, thick] (0,-1) -- (2,-1) -- (3,0);
   \draw[-] (0,0) -- (3,0) node[right=10pt] {\normalsize $i$};
   \draw[->] (0,-1.5) -- (0,0.5) node[left=5pt] {\normalsize $\rho(i)$};
  \end{tikzpicture}
 \]
 The roots of $p$ can be calculated from $\rho$ by the formula $a_i=\exp\big(\rho(i)-\rho(i-1)\big)$; cf.\ \cite[Thm.\ 4.1]{Baker-Lorscheid18b} for details. This yields the roots $a_1=a_2=1$ and $a_3=r$ of $p$. We encourage the reader to convince herself or himself that indeed $p\in(T+1)\hyperdot(T+1)\hyperdot (T+r)$.

 Thus we see that $a=r=a_3$ is a root of $p$ of multiplicity $1$. We are prepared to execute the algorithm to determine $q=\sum d_iT^i$. In our example, we have $k=3$ and $m=1$. Thus only steps \eqref{algo2} (for $i=0,1$) and \eqref{algo3} (for $i=2$) of the algorithm apply to determine the coefficients $d_i$ of $q$. We calculate for increasing $i=0,\dotsc,2$:
 \[
  d_0 = a^{-1}c_0 = 1, \quad d_1 = \max\{a^{-1}c_1,\, a^{-1}d_0\} = r^{-1}, \quad d_2 = c_3 = 1.
 \]
 Thus we find that $q=T^2+r^{-1}T+1$ is a divisor of $p$ by $T+r$. Once again, we encourage the reader to verify that indeed $p\in(T+r)\hyperdot q$.
 
 In order to exhibit some of the earlier mentioned effects that occur in the tropical setting and differ from the situation of polynomials over a field, we analyse this example in more detail. To begin with, we determine all polynomials $q'=\sum d_i'T^i$ that satisfy $p\in (T+a)\hyperdot q'$, i.e.\
 \[
  c_3 \ = \ d'_2, \quad c_2 \ \in \ ad'_2 \hyperplus d'_1, \quad c_1 \ \in \ ad'_1 \hyperplus d'_0, \quad c_0 \ = \ ad'_0.
 \]
 The first and last condition imply that $d'_2=c_3=1$ and $d'_0=a^{-1}c_0=1$, respectively. Using reversibility \ref{HG6}, the two middle conditions can be rewritten as 
 \[
  d'_1 \ \in \ (ad'_2)\hyperplus c_2 \ = \ [0,r] \quad\text{and}\quad  d'_1 \ \in \ (a^{-1}d'_0)\hyperplus (a^{-1}c_1) \ = \ [0,r^{-1}],
 \]
 which are simultaneous satisfied if and only if $d'_1\in [0,r^{-1}]$. Thus the divisors of $p$ by $T+a$ are precisely the polynomials of the form $q_s=T^2+sT+1$ with $s\in [0,r^{-1}]$.
 
 This shows that there are several divisors $q_s$ of $p$ by $T+r$. Note that $q=q_{r^{-1}}$ is maximal among all divisors. It also shows that the naive attempt to find a divisor $\tilde q=\sum \tilde d_iT^i$ in terms of elementary symmetric polynomials $\tilde d_i=\sigma_{2-i}(a_1,a_2)$, i.e.\ 
 \[
  \tilde q \ = \ T^2 + \max\{a_1,a_2\}T + a_1a_2 \ = \ T^2+T+1
 \]
 fails to provide a divisor of $p$ by $T+r$, in contrast to the situation over a field. 
 
 This example also shows that we cannot replace step \eqref{algo2} of the division algorithm neither by \eqref{algo1} nor solemnly by \eqref{algo3}. To wit, step \eqref{algo1} produces the coefficients
 \[
  d_2=c_3=1, \quad d_1=\max\{c_2,ad_2\}=r, \quad d_0=\max\{c_1,ad_1\}=r^2,
 \]
 and step \eqref{algo3} produces the coefficients
 \[
  d_2=c_3=1, \quad d_1=a_3c_3=r, \quad d_0=a_2a_3c_3=r,
 \]
 which both fail to provide a divisor $q=\sum d_iT^i$ of $p$ by $T+ar$.
 
 An example where step \eqref{algo2} fails to provide a divisor of $p$ by $T+a$ is the polynomial $p=T^3+rT^2+T+r$ with $r\in(0,1)$ and the root $a=r$. To wit, the roots of $p$ are $a_1=r$ and $a_2=a_3=1$. We have $p\in(T+r)\hyperdot q$ if and only if $q=T^2+sT+1$ with $s\in[0,r]$. But the steps in \eqref{algo2} produce the polynomial $r^{-2}T^2+r^{-1}T+1$, which is not a divisor of $p$ by $T+r$.
%\end{ex}

%%%%%%%%%%%%%%%%%%%%%%%%%%%%%%%%%%%%%%%%%%%%%%%%%%%%%%%%%%%%%%%%%%%%%%%%%%%%%%%%%%%%%%%%%%%%%%%%%%%%%%%%%%%%%%%%%%%%%%%%%%%%%%%%%%%%%%%%%%%%%%%%%%%%%%%%%%%%%%%%%%%%%%

\subsection{The proof of Theorem \ref{thm: division algorithm for tropical polynomials}}
\label{section: proof of Theorem A}

Let $a,a_1,\dotsc,a_n\in\T$ and the polynomials $p=\sum c_iT^i$ and $q=\sum d_iT^i$ be as in Theorem \ref{thm: division algorithm for tropical polynomials}, i.e.\ $c_n^{-1}p\in\hyperprod(T+a_i)$ is the unique factorization into linear terms with $a_1\leq\dotsb\leq a_n$, the nonzero element $a$ is a root of $p$ and the $d_i$ are defined by the algorithmic steps \eqref{algo1}--\eqref{algo3}. In this section, we prove Theorem \ref{thm: division algorithm for tropical polynomials}, i.e.\ 
 \[
  c_n \ = \ d_{n-1}, \qquad c_0 \ = \ ad_0 \qquad \text{and} \qquad c_i \in (ad_i)\hyperplus d_{i-1} \qquad \text{for} \qquad i=1,\dotsc,n-1.
 \]
% Since $p\in(T+a)\hyperdot q$ if and only if $c_n^{-1}p\in(T+a)\hyperdot (c_n^{-1}q)$, we can assume without loss of generality that $c_n=1$ and $p$ is monic.

If $k\leq n-m$, then $c_n=d_{n-1}$ follows immediately from the definition in step \eqref{algo1}. If $k=n-m+1$, then $d_{i-1}$

The relation $c_n=d_{n-1}$ follows immediately from the definition in step \eqref{algo1} if $k\leq n-m$ and the definition in step \eqref{algo1} if $k= n-m+1$. The relation $c_0=ad_0$ follows immediately from the definition in step \eqref{algo2} if $k\geq 2$. If $k=1$, then $a=a_1$ and according to the definition in step \eqref{algo3},
\[
 ad_0 \ = \ aa_2\dotsb a_nc_n \ = \ a_1\dotsb a_nc_n \ = \ c_0,
\]
as desired. 

Since $c_i \in (ad_i)\hyperplus d_{i-1}$ if and only if the minimum among $c_i$, $ad_{i}$ and $d_{i-1}$ occurs twice, the relation $c_i \in (ad_i)\hyperplus d_{i-1}$ is satisfied for $i=k+m,\dotsc,n-1$ and $i=1,\dotsc,k-2$ by the very definition of $d_{i-1}$ in \eqref{algo1} and $d_i$ in \eqref{algo2}, respectively. 

Since $a=a_{i+1}$ for $i=k-1,\dotsc,k+m-2$, we have that
\[
 ad_i \ = \ aa_{i+2}\dotsb a_{n} \ = \ d_{i-1}
\]
for $i=k,\dotsc,k+m-2$, and thus $ad_i\hyperplus d_{i-1}=[0,d_{i-1}]$. The relation $p\in\hyperprod(T+a_i)$ means that
\[
 c_i \ \leq \ a_{i+1}\dotsb a_{n}c_n \ = \ d_{i-1},
\]
and thus $c_i\in ad_i\hyperplus d_{i-1}$ for $i=k,\dotsc,k+m-2$, as desired.
 
We are left with $i=k-1$ and $i=k+m-1$, which are the critical cases that exhibit the compatibility between the different steps in the division algorithm. 
 
We begin with the case $i=k-1$. Since $a_{k-1}<a_{k}$, we have $c_{k-1}=a_k\dotsb a_nc_n$. By the definition in step \eqref{algo3} and since $a=a_{k}$, we have 
\[
 ad_{k-1} \ = \ a_{k}a_{k+1}\dotsb a_{n}c_n \ = \ c_{k-1}.
\]
If we can show that $d_{k-2}\leq a_k\dotsb a_nc_n$, then we obtain $c_{k-1}\in ad_{k-1}\hyperplus d_{k-2}$ as desired.
 
We claim that $d_j\leq a_{j+2}\dotsb a_nc_n$ for $j=0,\dotsc,k-2$, which we will prove by induction on $j$. The case $j=k-2$ is the missing inequality to conclude the proof of the case $i=k-1$. For $j=0$, we have indeed that $d_0=a^{-1}c_0\leq a_2\dotsb a_nc_n$ since $a_1\leq a$. For $j=1,\dotsc k-2$, we have $a_{j+1}\leq a_k=a$ and thus $a^{-1}a_{j+1}\dotsb a_nc_n\leq a_{j+2}\dotsb a_nc_n$. Since $c_j\leq a_{j+1}\dotsb a_nc_n$ by our assumptions and $d_{j-1}\leq a_{j+1}\dotsb a_nc_n$ by the inductive hypothesis, we get
\[
 d_j \ = \ \max\{a^{-1}c_j,\, a^{-1}d_{j-1}\} \ \leq \ a^{-1} a_{j+1}\dotsb a_nc_n \leq a_{j+2}\dotsb a_nc_n,
\]
which verifies our claim and concludes the proof of the case $i=k$.

We turn to the case $i=k+m-1$. By the definition in step \eqref{algo3}, $d_{k+m-2}=a_{k+m}\dotsb a_nc_n$. Since $a_{k+m-1}<a_{k+m}$, we have $c_{k+m-1}=a_{k+m}\dotsb a_nc_n=d_{k+m-2}$. If we can show that $ad_{k+m-1}\leq a_{k+m}\dotsb a_nc_n$, then we obtain the desired relation $c_{k+m-1}\in ad_{k+m-1}\hypersum d_{k+m-2}$.
 
We claim that $ad_{j}\leq a_{j+1}\dotsc a_nc_n$ for $j=n-1,\dotsc,k+m-1$, which we will prove by induction on $j$ (in decreasing order). The case $j=k+m-1$ is the missing inequality to conclude the proof of the case $i=k+m-1$. For $j=n-1$, we have $d_{n-1}=c_n$ and $a \leq a_n$. Thus $ad_{n-1}\leq a_nc_n$, as claimed. For $l=n-2,\dotsc,k+m-1$, we have $a\leq a_{j+1}$. Since $c_{j+1}\leq a_{j+2}\dotsb a_nc_n$ by our assumptions and $ad_{j+1}\leq a_{j+2}\dotsc a_nc_n$ by the inductive hypothesis, we get
\[
 ad_{j} \ = \ a\cdot \max\{c_{j+1},\,ad_{j+1}\} \ \leq \ a_{j+1} a_{j+2}\dotsb a_nc_n,
\]
as claimed. This concludes the proof of Theorem \ref{thm: division algorithm for tropical polynomials}. \qed

%%%%%%%%%%%%%%%%%%%%%%%%%%%%%%%%%%%%%%%%%%%%%%%%%%%%%%%%%%%%%%%%%%%%%%%%%%%%%%%%%%%%%%%%%%%%%%%%%%%%%%%%%%%%%%%%%%%%%%%%%%%%%%%%%%%%%%%%%%%%%%%%%%%%%%%%%%%%%%%%%%%%%%
%%%%%%%%%%%%%%%%%%%%%%%%%%%%%%%%%%%%%%%%%%%%%%%%%%%%%%%%%%%%%%%%%%%%%%%%%%%%%%%%%%%%%%%%%%%%%%%%%%%%%%%%%%%%%%%%%%%%%%%%%%%%%%%%%%%%%%%%%%%%%%%%%%%%%%%%%%%%%%%%%%%%%%

\section{Factorizations of sign polynomials}
\label{section: sign polynomials}

A \emph{sign polynomial} is a polynomial over the sign hyperfield $\S$. In this section, we classify all irreducible sign polynomials and show that the sign hyperfield fails to have the unique factorization property. Still it admits a division algorithm for the division of sign polynomials by linear terms, in analogy to the division algorithm for tropical polynomials.

%%%%%%%%%%%%%%%%%%%%%%%%%%%%%%%%%%%%%%%%%%%%%%%%%%%%%%%%%%%%%%%%%%%%%%%%%%%%%%%%%%%%%%%%%%%%%%%%%%%%%%%%%%%%%%%%%%%%%%%%%%%%%%%%%%%%%%%%%%%%%%%%%%%%%%%%%%%%%%%%%%%%%%

\subsection{Classification of the irreducible polynomials}
\label{subsection: classification of the irreducible polynomials}

%In this section, we classify the irreducible sign polynomials. 
Since a sign polynomial $p$ is irreducible if and only if $ap$ is irreducible for any $a\in\S^\times$, we can restrict our attention to monic irreducible sign polynomials.

\begin{thm}\label{thm: classification of irreducible sign polynomials}
 The monic irreducible sign polynomials are $T$, $T-1$, $T+1$ and $T^2+1$.
\end{thm}

\begin{proof}
 It is clear that every linear polynomial is irreducible, cf.\ section \ref{subsection: Factorization}. Thus $T$, $T-1$ and $T+1$ are precisely the monic irreducible polynomials of degree $1$. 
 
 Given a monic irreducible sign polynomial $p=\sum c_iT^i$, let $\hat c_i\in \R$ be real numbers with $\sign(\hat c_i)=c_i$. By Lemma \ref{lemma: factorizations are preserved under morphisms}, the monic real polynomial $\hat p=\sum \hat c_iT^i$ is irreducible as well. We know that the monic irreducible real polynomials are either linear or quadratic with positive constant term $\hat c_0>0$. Thus if $p$ is not linear then it must be of the form $T^2+aT+1$ for some $a\in\S$.
 
 By Lemma \ref{lemma: irreducibility of quadratic polynomials}, a quadratic polynomial is irreducible if and only if it does not have a root. We can verify this for all polynomials of the form $T^2+aT+1$:
 while $T^2+T+1$ has $-1$ as a root and $T^2-T+1$ has $1$ as a root, $T^2+1$ is the only quadratic polynomial of this shape that does not have a root. This completes our classification of the monic irreducible sign polynomials.
\end{proof}

%%%%%%%%%%%%%%%%%%%%%%%%%%%%%%%%%%%%%%%%%%%%%%%%%%%%%%%%%%%%%%%%%%%%%%%%%%%%%%%%%%%%%%%%%%%%%%%%%%%%%%%%%%%%%%%%%%%%%%%%%%%%%%%%%%%%%%%%%%%%%%%%%%%%%%%%%%%%%%%%%%%%%%

\subsection{The failure of unique factorization}
\label{subsection: the failure of unique factorization}

It is easy to see that the unique factorization property holds for sign polynomials of degree $\leq 2$. The following example shows that this property fails from degree $3$ on.

Consider the sign polynomial $p=T^3+T^2+T+1$. Then $-1$ is a root of $p$, i.e.\ $0 \in p(-1)$. Thus there is a polynomial $q=d_2T^2+d_1T+d_0$ such that $p\in (T+1)\hyperdot q$, which is equivalent to
\[
 d_0=1, \quad d_2=1 \quad \text{and} \quad 1 \in 1\hyperplus d_1.
\]
The equation $1\in 1\hyperplus d_1$ is is true for all $d_1\in \S$, which means that $p$ is an element of all the three hyperproducts
\[
 (T+1)\hyperdot (T^2+1), \quad (T+1)\hyperdot (T^2+T+1) \quad \text{and} \quad (T+1)\hyperdot (T^2-T+1).
\]
The factors $T^2\pm T+1$ factorize into $T^2+T+1 \in (T+1)\hyperdot (T+1)$ and $T^2-T+1 \in (T-1)\hyperdot (T-1)$, respectively. The factor $T^2+1$ is irreducible. Thus we find the three different factorizations
\[
 (T+1)\hyperdot (T^2+1), \quad (T+1)\hyperdot (T+1)\hyperdot (T+1) \quad \text{and} \quad (T+1)\hyperdot \Big((T-1)\hyperdot (T-1)\Big)
\]
of $T^3+T^2+T+1$.
%, so $p \in (T+1)\hyperdot (T+1)\hyperdot (T+1)$. Since $T+1$ and $T^2+1$ are irreducible polynomials, it follows that $p$ has two different factorization into irreducible factors.

In fact, this example shows that sign polynomials cannot have the unique factorization property with respect to any concept of factorization that is preserved under morphisms, in the sense of Lemma \ref{lemma: factorizations are preserved under morphisms}. Indeed, the sign map $\sign:\R\to \S$ maps both real polynomials $T^3+T^2+T+1=(T+1)(T^2+1)$ and $T^3+3T^2+3T+1=(T+1)^3$ to $p$.

%%%%%%%%%%%%%%%%%%%%%%%%%%%%%%%%%%%%%%%%%%%%%%%%%%%%%%%%%%%%%%%%%%%%%%%%%%%%%%%%%%%%%%%%%%%%%%%%%%%%%%%%%%%%%%%%%%%%%%%%%%%%%%%%%%%%%%%%%%%%%%%%%%%%%%%%%%%%%%%%%%%%%%

\subsection{A division algorithm}
\label{subsection: division algorithm for sign polynomials}

In spite of the failure of unique factorization, there is still an algorithmic way to determine a divisor of a sign polynomial by a linear term. Such a division algorithm was already exhibited in the proof of Theorem 3.1 in \cite{Baker-Lorscheid18b} for a restricted class of sign polynomials. In the following, we describe an extension of this division algorithm that applies to all sign polynomials.

As a preliminary consideration, we observe that if $a=0$ is a root of a sign polynomial $p=\sum c_iT^i$, then $c_0=0$ and $q=\sum c_{i+1}T^i$ is the unique sign polynomial such that $p\in T\hyperdot q$. Thus it suffices to describe the division algorithm for nonzero roots $a\in\S^\times=\{\pm 1\}$ only.

\begin{thm}\label{thm: division algorithm for sign polynomials}
 Let $p=\sum c_iT^i$ be a sign polynomial of degree $n$ with root $a\in\{\pm 1\}$. Define 
 \[
  l \ = \ \min\big\{\, i\in\N\,\big|\,c_i\neq 0\,\big\} \qquad \text{and} \qquad k \ = \ \min\big\{\,i\in\N\,\big|\, c_{i+1}= -a^{i+1-l} c_l\, \big\}.
 \]
 Define recursively for $i=n-1,\dotsc,0$ (in decreasing order)
 \begin{align}
  \label{sign1} d_i \ &= \ c_{i+1}        && \text{if $c_{i+1}\neq 0$ and $i>k$;}\\
  \label{sign2} d_i \ &= \ ad_{i+1}       && \text{if $c_{i+1}=0$ and $i>k$;}\\
  \label{sign3} d_i \ &= \ -a^{i+l-1} c_l && \text{if $l\leq i\leq k$;}\\
  \label{sign4} d_i \ &= \ 0              && \text{if $0\leq i<l$.}
 \end{align}
  Then $p\in(T-a)\hyperdot q$ for $q=\sum d_iT^i$.
\end{thm}

\begin{proof}
 Once we have proven the theorem for the root $a=1$ of $p=\sum c_iT^i$, we can derive the division algorithm for the root $-1$ by applying the division algorithm for $a=1$ to $p(-T)=\sum (-1)^ic_iT^i$ and using that $p\in\big(T-(-1)\big)\hyperdot q$ if and only if $p(-T)\in -(T-1)\hyperdot q(-T)$. Thus the case $a=-1$ follows by a straight forward calculation from the case $a=1$. 
 
 We proceed with the proof for $a=1$. Recall that $p\in(T-1)\hyperdot q$ if and only if 
 \[
  c_n \ = \ d_{n-1}, \qquad c_0 \ = \ -d_0 \qquad \text{and} \qquad c_i \in (-d_i)\hyperplus d_{i-1} \qquad \text{for} \qquad i=1,\dotsc,n-1.
 \]
 
 We begin with $c_n=d_{n-1}$. If $k<n-1$, then $d_{n-1}=c_n$ by \eqref{sign1} since $c_n\neq0$. If $k=n-1$, then $c_n=-c_l$ by the definition of $k$ and thus $d_{n-1}=-c_l=c_n$ by \eqref{sign3}. Thus $c_n=d_{n-1}$, as desired.
 
 We proceed with $c_0=-d_0$. If $l=0$, then $d_0=-c_0$ by \eqref{sign3}. If $l>0$, then $c_0=0$ and $d_0=0=c_0$ by \eqref{sign4}. Thus $c_0=-d_0$, as desired.
 
 We proceed with $c_i \in (-d_i)\hyperplus d_{i-1}$ for $1\leq i\leq n-1$. If $1\leq i<l$, then $c_i=0$ and $d_{i-1}=d_i=0$ by \eqref{sign4}. Thus $c_i \in (-d_i)\hyperplus d_{i-1}$, as desired.
 
 If $i=l$, then $c_l\neq 0$, $d_l=-c_l$ by \eqref{sign3} and $d_{l-1}=0$ by \eqref{sign4}. Thus $c_i \in (-d_i)\hyperplus d_{i-1}$, as desired.
 
 If $l< i\leq k$, then $d_{i-1}=d_i=-c_l\neq0$ by \eqref{sign3} and the definition of $l$. Thus $c_i \in (-d_i)\hyperplus d_{i-1}$, as desired.
 
 If $i=k+1$, then $c_{k+1}=-c_l\neq0$ by the definitions of $k$ and $l$. Thus $d_{k}=-c_l\neq 0$ by \eqref{sign3}, and $c_i \in (-d_i)\hyperplus d_{i-1}$, as desired.
 
 If $k+1< i \leq n-1$ and $c_{i}\neq0$, then $d_{i-1}=c_{i}\neq 0$ by \eqref{sign1}. If $k+1< i \leq n-1$ and $c_{i}=0$, then $d_{i-1}=d_{i}$ by \eqref{sign2}. Thus in both cases $c_i \in (-d_i)\hyperplus d_{i-1}$, as desired. This concludes the proof of the theorem.
\end{proof}

%%%%%%%%%%%%%%%%%%%%%%%%%%%%%%%%%%%%%%%%%%%%%%%%%%%%%%%%%%%%%%%%%%%%%%%%%%%%%%%%%%%%%%%%%%%%%%%%%%%%%%%%%%%%%%%%%%%%%%%%%%%%%%%%%%%%%%%%%%%%%%%%%%%%%%%%%%%%%%%%%%%%%
%%%%%%%%%%%%%%%%%%%%%%%%%%%%%%%%%%%%%%%%%%%%%%%%%%%%%%%%%%%%%%%%%%%%%%%%%%%%%%%%%%%%%%%%%%%%%%%%%%%%%%%%%%%%%%%%%%%%%%%%%%%%%%%%%%%%%%%%%%%%%%%%%%%%%%%%%%%%%%%%%%%%%

\begin{small}
 \bibliographystyle{plain}
 \bibliography{hyperfield}
\end{small}

\end{document}